\documentclass[12pt]{amsart}
\usepackage {amsmath, amssymb, a4wide, epsfig, enumerate, psfrag}
\usepackage[latin1]{inputenc}
\usepackage[english]{babel}
\usepackage[active]{srcltx}

\newcommand{\norm}[1]{\left\Vert#1\right\Vert}

\date{\today}

\author{Bertrand Deroin}
\thanks{ }
\address{CNRS -- Laboratoire AGM -- Universit\'e de Cergy-Pontoise}
\email{bertrand.deroin@u-cergy.fr}

\title[Locally discrete expanding groups]{Locally discrete expanding groups of analytic diffeomorphisms of the circle}

\newtheorem{prop} {Proposition} [section]

\newtheorem{definition}[prop] {Definition}
\newtheorem{lemma}[prop] {Lemma}
\newtheorem{corollary}[prop]{Corollary}
\newtheorem{proposition}[prop]{Proposition}

\theoremstyle{remark}

\newtheorem{remark}[prop]{Remark}

\def\S{S}

\begin{document}

\begin{abstract} We show that a finitely generated subgroup of \(\text{Diff}^\omega ({\bf S}^1) \) that is expanding and locally discrete in the analytic category  is analytically conjugated to a uniform lattice in \( \widetilde{\text{PGL}}^k _2({\bf R}) \) acting on the \(k\)-th covering of \({\bf R}P^1\) for a certain integer \(k>0\).
\end{abstract}

\maketitle


\section{Introduction} In the study of the dynamics of finitely generated groups acting by analytic diffeomorphisms on the circle (or more generally in analytic unidimensional dynamics) the dichotomy discreteness versus non discreteness is very useful and important. Many interesting dynamical properties can be easily established in the non locally discrete regime, for instance concerning ergodicity or minimality of the action. However in the locally discrete regime things are not completely understood yet, even if a conjectural classification is expected, see e.g. the survey \cite{DFKN}.
   
The goal of this note is to provide such a classification under the additional assumption that the action is expansive, as announced in \cite{DFKN}. Expansive means that for every point of the circle there exists an element of the group whose derivative at that point is greater than \(1\) in modulus. Our main result (Corollary \ref{c: classification}) shows that up to analytic conjugacy, only cocompact lattices of the finite cyclic coverings of \(\text{PGL}_2 ({\bf R}) \) acting on the corresponding finite cyclic covering of the real projective line \({\bf R}P^1\) are at the same time expansive and locally discrete in the analytic category. The precise definitions of expansiveness and local discreteness in the analytic category are exposed in sections \ref{s: expanding property} and \ref{s: local discreteness} respectively.


This result is part of a more general result concerning the dynamics of pseudo-groups of holomorphic maps on Riemann surfaces having both local discreteness and hyperbolicity properties. However, its proof in the particular case of the circle group action is considerably simpler (essentially because of the use of a combination of the convergence group theorem by Gabai \cite{Gabai} and Casson-Jungreis \cite{CJ}, and of the differentiable rigidity theory of Fuchsian groups by Ghys \cite{Ghys_rigidite}), and deserves a special interest for the theory  of circle group actions. 

\vspace{0.5cm} 

{\bf Organisation of the paper.} Section \ref{s: geometric group theory} is devoted to review some aspects of the theory of hyperbolic groups that will be needed in our argument. In sections \ref{s: expanding property} and \ref{s: distortion} (resp. \ref{s: local discreteness}) we present the definition of expansiveness (resp. local discreteness) that is assumed in our main result. Section \ref{s: convergence} is devoted to the main technical tool of our method, namely the convergence property of the lines of expansion. The last three sections \ref{s: hyperbolicity}, \ref{s: Fuchsian group} and \ref{s: differentiable rigidity} are devoted to the proof of our main result: the Corollary \ref{c: classification}. 

\vspace{0.5cm}

{\bf Acknowledgments --} I express my gratitude to the team \cite{AFKMMNT} who encouraged me to write this note. I particularly thank Michele Triestino for his careful reading. 

\section{Index of notations}

We denote by greek letters the distances in the Cayley graphs we consider.

\begin{itemize}
\item \({\bf S}^1={\bf R}/{\bf Z}\)
\item \( x, y , \ldots\) points of \({\bf S}^1\)
\item \(B(x,r)\) : ball of radius \(r\) centered at \(x\) in \({\bf C}/{\bf Z}\) for the euclidean metric
\item \(G\) subgroup of \(\text{Diff}^\omega({\bf S}^1)\) 
\item \(\S\) symmetric set of generators of \(G\) 
\item \(\norm{g}\) minimal number of generators needed to write \(g\)
\item \(\delta\) constant of hyperbolicity for the Cayley graph associated to the pair \((G,S)\)
\item \(f,g,\ldots\) elements of \(G\)
\item \(s\) element of \(\S\)
\item \( \{ E_m^x\} _{m\geq 0}\) line of expansion at \(x\) 
\item \( x_m= E_m^x (x) \)
\item \(Dg\) derivative of the map \(g\) 
\item \(D_x : G\times G\rightarrow {\bf R}\) derivative cocycle
\item \(\kappa(g,E)\) distortion of the map \(g\) on the set \(E\) 
\item \(a,b \) constant appearing in Lemma \ref{l:discreteness} (controlling size of derivatives) 
\item \( \underline{c} >0 \) constant of uniform expansion of derivatives along the set \(S\)
\item \( \overline{c}>0\) constant controlling logarithms of derivatives of elements of \(S\), 
\item \(\gamma_i\)'s constant of quasi-isometry appearing in Proposition \ref{p:line of expansion}
\item \(\Omega : {\bf S}^1\rightarrow \partial G\) equivariant map defined by equation \eqref{def: map to boundary}
\item \( \phi : \partial G\rightarrow {\bf R}P^1\) map that is \(\rho\)-equivariant
\item \(\Phi = \phi \circ \Omega\)
\item for every positive integer \(k\), \(\widetilde{\text{PGL}_2^k({\bf R})} \) is the \(k:1\) cyclic covering of \( \text{PGL}_2({\bf R})\), acting on  the \(k:1\) cyclic covering \(\widetilde{{\bf R}P^1}^k\) of \({\bf R}P^1\)
\item \( U\subset {\bf S}^1 \times {\bf R} P ^1 \) complement of the graph of \(\Phi\) 
\item \(K\subset U\times {\bf R}\) compact subset defined by equation \eqref{eq: compact set}
\item \( (x,z,t)\) coordinates of a point in \( U\times {\bf R}\)
\item \( \Delta\subset U\times {\bf R}\) fundamental domain for the \(G\)-action on \(U\) defined by equation \eqref{eq: action}
\item  \(M\) quotient of \(U\) by \(G\) 
\item \(V=\frac{\partial}{\partial t} \) vector fields on \(M\)
\item \(\mathcal F^\pm \): weak stable foliations of \(V\) on \(M\) 
\item \(D\) developing map
\end{itemize}

\section{Preliminaries of geometric group theory} \label{s: geometric group theory}

We review some notions of geometric group theory that will be useful for our argument. 

Let \(G\) be a finitely generated group. Given a finite symmetric generating subset \(S\subset G\), we associate the norm \(\norm{g}\) of an element \(g\in G\) as being the minimum number of elements of \(S\) that is needed to write \(g\).  

The {\bf Cayley graph} of the pair \((G,S)\) is the non oriented graph whose vertices are the elements of \(G\) and the edges are the pairs \( \{g, sg\}\) with \(g\in G\) and \(s\in S\). The group \(G\) acts naturally on its Cayley graph by right multiplications. The combinatorial distance \(d\) associated to this graph -- namely, the one defined as the minimum number of edges one has to cross to go from a vertex to another one -- is given by the formula \( d(g_1,g_2) := \norm{g_2 g_1^{-1}}\). Another finite symmetric generating set gives rise to another graph whose set of vertices is \(G\), for which the identity map is {\bf bi-Lipschitz} with respect to the associated distances.

The group \(G\) is called {\bf hyperbolic} if its triangles are thin, in the sense that there exists a constant \(\delta>0\) such that for any triple of points \( g_1,g_2,g_3\in G\), and any collections of geodesics \([g_1,g_2]\), \([g_2,g_3]\) and \([g_3,g_1]\) between these points, we have that 
\begin{equation} \label{eq: delta thin} [g_1,g_3 ] \subset [g_1,g_2]^\delta \cup [g_2,g_3]^\delta \end{equation}
where \(A^\delta\) is the set of points at a distance from a point of \(A\) less than \(\delta\). A triangle in a graph satisfying an inequality such as \eqref{eq: delta thin} is called {\bf \(\delta\)-thin}.

A {\bf geodesic ray} parametrized by an interval \(I\subset {\bf Z}\) is a sequence \( \{g_n\}_{n\in I}\) of elements of \(G\) such that \(d (g_k, g_l) =|k-l|\) for every \(k,l\in I\). The set of equivalence classes of geodesic rays parametrized by \({\bf N}\) up to bounded Hausdorff distance, is the geometric boundary of \(G\), and is denoted by \( \partial G\). It is equipped with the quotient of the topology of simple convergence. In the case where \(G\) is Gromov hyperbolic, this topological space is a compact metric space. Moreover, in this latter case, the group \(G\) acts naturally on its boundary by homeomorphisms, and the action is minimal unless \(G\) is virtually cyclic.

Given constants \(\alpha \geq 1\) and \(\beta\geq 0 \), an \((\alpha ,\beta)\)-{\bf quasi-geodesic} on \(G\) is a sequence \(\{g_n\}_{n\in {\bf N}} \) of elements of \(G\) such  that for every \(m,n\in {\bf N}\), we have \begin{equation}\label{eq: quasi-geodesic} \alpha ^{-1} |n-m| - \beta \leq d(g_m,g_n) \leq \alpha |n-m| +\beta.\end{equation}
We recall that there exists a constant \(\eta\) such that two \( (\alpha, \beta) \)-quasi-geodesics having the same extremities are bounded appart in Hausdorff distance by less than \(\eta\). As a consequence of this, up to enlarging the constant  \(\delta\), any \((a,b)\)-quasi-geodesic triangle is \(\delta \)-thin. 

The {\bf Gromov product} is defined by
\begin{equation}\label{eq: Gromov product} (h_1,h_2)_g := \frac{d(g,h_1) + d(g,h_2) - d(h_1,h_2)}{2} , \end{equation}
for every \(g, h_1, h_2\in G\). 
Its geometrical significance in a Gromov hyperbolic graph is that, up to some constant depending only on the constants of hyperbolicity of the graph, the Gromov product \( (h_1,h_2) _g\) is the distance from \(g\) to a (any) geodesic between \(h_1\) and \(h_2\). Given two points \(p,q\in \partial G\), this Gromov product can be extended to points of the boundary of \(G\) by the following formula 
\begin{equation} \label{eq: Gromov product boundary} (p,q) _g := \sup _{(h_m)_m, \ (k_n)_n} \ \limsup _{m,n\rightarrow \infty} \ (h_m, k_n)_g ,\end{equation}
where the first supremum is taken over all geodesics \( \{h_m\}_{m\in {\bf N}}\) and \( \{k_n\}_{n\in {\bf N}}\) that tend to \(p\) and \(q\) respectively. In fact, if \((h_m)_m\) and \((k_n)_n\) are geodesics tending to \(p\) and \(q\) respectively, then  if \(m,n\) are sufficiently large, the quantity \( (h_m, k_n)_g\) differs from  \( (p,q)_g\) by an additive term which is bounded by a constant that depends only on the hyperbolicity constants of \(G\). Then two points \(p,q\in \partial G\) are close to each other if and only if the Gromov product \( (p,q)_e\) is large.

\section{Expanding property and derivative cocycle} \label{s: expanding property}

\begin{definition}
A subgroup \(G\subset \text{Diff}^\omega({\bf S}^1) \) is {\bf {\em expanding}} if for every point \(x\in {\bf S}^1\), there exists an element \( g\in G\) such that  \( \log | Dg(x) | >0\). 
\end{definition}

\begin{remark}\label{r: uniformity} By compacity of \({\bf S}^1\), if \(G\subset \text{Diff}^\omega({\bf S}^1) \) is a subgroup which is expanding, there exists a finite subset \(S\subset G\) and a constant \( \underline{c} >0\) such that for every point \(x\in {\bf S}^1\) there exists \(s\in S\) such that \(\log |Ds (x)|\geq \underline{c} \).  If \(G\) is finitely generated, we can assume furthermore that \(S\) is a symmetric generating set. \end{remark}

\begin{definition}
If \(G\) is a subgroup of \(\text{Diff}^\omega ({\bf S}^1)\), for every \(x\in {\bf S}^1\), we define the {\bf {\em derivative cocycle}} \(D_x\) by the formula
\begin{equation}\label{eq: derivative cocycle} D _x (g_1,g_2) := \log |Dg_2 (x) |- \log|Dg_1(x)| = \log |D(g_2\circ g_1^{-1} ) (g_1(x) )|  .\end{equation}
\end{definition}

\begin{remark}\label{r: equivariance} The derivative cocyle satisfies the obvious equivariance 
\begin{equation} \label{eq: equivariance}  D _{g(x)} (g_1\circ g^{-1}, g_2\circ g^{-1} ) = D_x (g_1, g_2) . \end{equation}
for every \(x\in {\bf S}^1\), and \(g,g_1, g_2\in G\). 

Moreover, if \(G\) is generated by a finite symmetric subset \(S\), the derivative cocycle satisfies 
\begin{equation}\label{eq: bound cocycle} D_x (g_1, g_2) \leq \overline{c} \ d( g_1, g_2) , \end{equation} where  \begin{equation}\label{eq: max expansion} \overline{c} := \max_{x\in {\bf S}^1 , s\in S} \log |Ds(x)| .\end{equation} 
\end{remark} 

\begin{definition}\label{def: line of expansion} Let \(G\subset \text{Diff}^\omega({\bf S}^1)\) be a finitely generated subgroup,  \(S\subset G\) a finite symmetric generating subset, \(c >0\),  and \(x\in {\bf S}^1\). A \( c \)-{\bf {\em line of expansion}} relative to the point \(x\) is a sequence \(\{E_n^x \}_{n\geq 0} \) of elements of \(G\) such that for every \(n> 0\) one can write \( E_n^x= s_n^x E_{n-1}^x\) for some \(s_n^x\in S\), and for every \(0\leq m\leq n\) we have 
\begin{equation}\label{eq: line of expansion} D_x ( E_{m}^x, E_n^x) \geq c (n-m) .\end{equation}
\end{definition}

\begin{remark} Observe that if \(S\) and \(\underline{c} \) are given by Remark \ref{r: uniformity}, for any \(0< c\leq \underline{c}\), for any point \(x \in {\bf S}^1\), and for every \(g\in G\), there exists a \( c\)-line of expansion relative to the point \(x\) starting at \(g\). In the sequel, we will choose either \(c=\underline{c}\) or \(c=\underline{c}/2\). \end{remark}

\begin{lemma}\label{l: quasi-geodesic}
Given a finitely generated group \(G\subset \text{Diff}^\omega({\bf S}^1)\), a finite symmetric generating subset \(S\subset G\), and a constant \(c >0\), there exists \(\alpha \geq 1\) such that  any \( c\)-line of expansion is a \( (\alpha, 0)\)-quasi-geodesic. 
\end{lemma}

\begin{proof} We clearly have \( d(E_m^x,E_n^x) \leq |m-n|\) for every \(m,n\) so the right hand side of \eqref{eq: quasi-geodesic} is satisfied with \(\alpha =1\) and \(\beta=0\). Moreover, for any \(m,n\in {\bf N}\) with \(m\leq n\), \eqref{eq: bound cocycle}, \eqref{eq: max expansion} together with \eqref{eq: line of expansion} show that 
\[  c(n-m)  \leq D_x (E_m^x, E_n^x) \leq \overline{c} \ d( E_m^x , E_n^x) \]
so that the left hand side of \eqref{eq: quasi-geodesic} is satisfied with \(\alpha = c/\overline{c} \) and \(\beta= 0\). 
 \end{proof}

\section{Bounded distortion along lines of expansion}\label{s: distortion}

\begin{definition} The {\bf {\em distortion}} of a differentiable analytic (resp. holomorphic) map \( g: U\rightarrow V\) between open subsets of \( {\bf S}^1\) (resp. \({\bf C}\) ) in restriction to a subset \( E\subset U\) is the quantity
\begin{equation} \label{eq: distortion} \kappa (g, E):= \max_{x,y\in E} \log \frac{| Dg (y) |}{|Dg(x) |}.\end{equation} 
\end{definition}

\begin{lemma}\label{l: distortion}
Let \( G \) be a subgroup of \(\text{Diff}^\omega ({\bf S}^1) \) generated by a finite symmetric set \(S\), and let \(c >0\) be some constant.  There exists \( r >0\) such that the following holds. For any \( c\)-line of expansion \( (E_n^x ) _{n\geq 0}\) relative to some \(x\in {\bf S}^1\), and for every \(n\in {\bf N}\), the element \( (E_n^x)^{- 1} \in G\) has a holomorphic univalent extension \(\widetilde{(E_n^x)^{-1}}\) to the ball \( B (x_n, r) \) of center \(x_n= E_n^x(x)\).
\end{lemma}

\begin{proof} Choose a real number $r>0$ small enough so that the following condition holds: every $s\in S$ extends as a univalent holomorphic map $\widetilde{s}$ defined on the \(r\)-neighborhood \( A_r\) of \({\bf S}^1\) in \({\bf C}/{\bf Z}\), and moreover for every \(y\in {\bf S}^1\)
\begin{equation}\label{eq: bound distortion} \kappa (\widetilde{s}, B(y,r)) \leq c.\end{equation}

Let \(0\leq m \leq n\) be an integer. Since \( (E_n^x ) _{n\geq 0}\) is a \(c\)-line of expansion at \(x\), we have that \( |D (s_m^x)^{-1} (x_m)| \leq e^{-c}\). So \eqref{eq: bound distortion} shows that for every \( y\in B(x_m,r)\) 
\[ |D (\widetilde{s_m^x)^{-1}} (y)| \leq e^c \cdot |D (s_m^x)^{-1} (x_m)|  \leq 1. \]
Hence, the ball \(B(x_m, r)\) is sent by \( \widetilde{(s_m^x)^{-1}}\) inside the ball \( B(x_{m-1} , r)\). 

Since \( (E_n^x)^{-1} := (s_1 ^x)^{-1} \ldots  (s_{n}^x)^{-1}\), it has a holomorphic extension to \(B(x_n,r) \), which maps this latter to \(B(x,r)\).\end{proof}

To end this paragraph, recall the following distortion estimate, due to Koebe \cite{Koebe}:

\begin{lemma}[Koebe] 
\label{l:Koebe}
There is a constant $\kappa>0$ such that the following holds. Let $z_1,z_2$ be points of ${\bf C}$, $ r >0$ be a positive real number, and $g$ be a univalent non constant holomorphic map from $ B(z_1,r) $ to $ {\bf C}$ sending $z_1$ to $z_2$. Then for every \(0\leq r' \leq r /2\), 
\[ \kappa (g, B(z_1, r') ) \leq \kappa.\] 
Moreover, $g(B(z_1,r' ))$ contains the ball $B(z_2, r' e^{-\kappa} |Dg(z_1)|)$.
\end{lemma} 

\section{Local discreteness} \label{s: local discreteness}
For the next definition, we think of  the circle as being the one dimensional submanifold \({\bf S}^1= {\bf R}/{\bf Z}\) contained in the Riemann surface \({\bf C}/{\bf Z}\).   

\begin{definition}
A subgroup \( G\subset \text{Diff} ^\omega ({\bf S}^1)\) is {\bf {\em non locally discrete in the analytic category at a point \(x\in {\bf S}^1\)}}, if there exists a neighborhood \(V\subset {\bf C}/{\bf Z}\) of \(x\) and a sequence \(\{g_n\}_{n\geq 1}\) of elements of \(G\setminus \{\text{id}\} \) that extend as univalent holomorphic maps from \(V\) to \({\bf C}/{\bf Z}\) and whose extensions to \(V\) tend uniformly to the identity on \(V\) when \(n\) tends to infinity. Otherwise, \(G\) is {\bf {\em locally discrete in the analytic category at the point \(x\)}}.
\end{definition}

We have the following simple characterization of the local discreteness of $G$ everywhere in the analytic category. For every $x\in {\bf C}/{\bf Z}$, and every non negative real number $r$, we denote by $B(x,r)$ the ball of radius $r$ centered at $x$ in ${\bf C}/{\bf Z}$, for the euclidean distance.

\begin{lemma} \label{l:discreteness}
Let  $G\subset \text{Diff}^\omega({\bf S}^1) $ be a finitely generated subgroup, which is locally discrete in the analytic category at every point. Let \(S\) be a finite symmetric generating set for \(G\).  Then, given numbers $r>0$ and $ a < b$, there is an integer $\gamma\in {\bf N}^*$ such that the following holds. Let $y\in {\bf S}^1$ and $f\in   G$. Suppose that \(f\)  can be extended as a univalent holomorphic map \(\widetilde{f}\) defined on the ball $B(y , r)$  that satisfies $ a \leq \log |D \widetilde{f}| \leq b$ on $B(y, r)$. Then $f$ is a composition of at most $\gamma$ elements of $S$, namely \( \norm{f} \leq \gamma\).    
\end{lemma} 

\begin{proof} 
By contradiction, suppose that this is not true. Then there is a sequence $(f_k)_k$ of elements of $G$ and a sequence $(y_k)_k$ of points of ${\bf S}^1$ such that for every $k$,  $f_k$ has a holomorphic univalent extension \(\widetilde{f_k}\) defined on $B(y_k, r)$ whose derivative on this latter satisfies \(a \leq \log |D\widetilde{f_k}|\leq b\), and whose word-length tends to infinity. Taking a subsequence if necessary, we can suppose that $(y_k)_k$ converges to a point $y\in {\bf S}^1$ when $k$ tends to infinity. Then, the maps $\widetilde{f_k}$ are defined on the ball $B(y, r /2)$ when $k$ is large enough. Since logarithm of derivatives are bounded on \( B(y,r/2)\), Montel's theorem shows that, taking a subsequence if necessary, the maps $\widetilde{f_k}$ converge uniformly to a holomorphic univalent map $\widetilde{f}: B(y,r/4)\rightarrow {\bf C}/{\bf Z}$. In particular, the maps $\widetilde{f_{k+1} }^{-1} \circ \widetilde{f_k}$ are well-defined on $B(y, r/8)$ if $k$ is large enough, and converge to the identity in the uniform topology when $k$ tends to infinity. The discreteness assumption implies that for $k$ large enough, $f_{k+1} = f_k$; this contradicts the fact that the word-length of $f_k$ tends to infinity.  
\end{proof} 

The following result shows that for an expanding group of analytic diffeomorphisms of the circle, being locally discrete in the analytic category somewhere is the same notion as being locally discrete in the analytic category everywhere.

\begin{lemma}\label{l: locally discrete somewhere vs everywhere}
If a subgroup \(G\subset \text{Diff}^\omega ({\bf S}^1)\) is expanding, then it is either non locally discrete in the analytic category everywhere, or locally discrete in the analytic category everywhere. 
\end{lemma}

\begin{proof}
The set of points where the group is non locally discrete in the analytic category is obviously an open set. Its complement, the set where the group is locally discrete in the analytic category, is a closed invariant subset \(\hat{\Lambda}\). Suppose that \(\hat{\Lambda}\) is non empty. 

We claim that \(\hat{\Lambda}\) does not contain a finite orbit in view of the expansiveness assumption. Indeed, any point whose orbit is finite is a hyperbolic fixed point of some element of \(G\) which takes the form \(f:z\mapsto \lambda z\) with \(|\lambda|<1\) in some linearizing coordinates.  Since the group is not virtually abelian (by expansiveness) the stabilizer of \(z=0\) contains an element tangent to the identity at \(z=0\) of the form \(g: z\mapsto z + az^k + \ldots \) with \(a\neq 0\) and \(k\geq 2\). Now, we have for a positive integer \(n\)
\begin{equation} \label{eq: renormalization} f^{-n} \circ g\circ f^{n}=  z + a\lambda ^{(k-1)n} z^k +\ldots \end{equation}
which shows that there exists a neighborhood of \(z=0\) in \({\bf C}\) where \( f^{-n} \circ g\circ f^n \) converges to the identity when \(n\) tends to \(+\infty\). This shows that the finite orbits are contained in the non locally discrete part. 

Hence \(\hat{\Lambda}\) contains an exceptional minimal subset \(\Lambda\). However, a result of Hector \cite{Hector1} shows that the stabilizers of the components of the complement of \(\Lambda\) are virtually cyclic. Hence every point in the complement of the exceptional minimal set is locally discrete in the analytic category (in fact in the compact open topology).  In particular, the group is locally discrete everywhere, and the result follows.
\end{proof}

In view of this result, we will call an expanding group {\bf locally discrete in the analytic category} if it is locally discrete in the analytic category at some point, and hence at every point.



\section{Convergence of lines of expansion in the Cayley graph}\label{s: convergence}

\begin{proposition}\label{p:line of expansion}
Let \( G\subset \text{Diff}^\omega({\bf S}^1) \) be a finitely generated subgroup which is expanding and locally discrete in the analytic category. Let \(S\) be a finite symmetric subset of generators of \(G\), 
\[ \overline{c} := \sup  _{x\in {\bf S}^1, s\in S}\log  |Ds(x)| ,\]
and let \(c>0\) be some constant. Then, there are constants $\gamma_1,\gamma_2,\gamma_3>0$ such that the following holds.  Let $x\in {\bf S}^1$,  and $(E_m ^x )_m$, $(F_n^x)_n$ be \( c\)-lines of expansion relative to the point \(x\). Then for every non negative integers $n,m\geq \gamma_1 d( E_0^x , F_0^x ) + \gamma_2$,  such that
\begin{equation}\label{eq:line of expansion} D_x (E_m^x, F_n^x ) \leq c,\end{equation}
we have $d( E_m^x,F_n^x)  \leq \gamma_3$. \end{proposition}

\begin{proof} Moving \(x\) to \(E_0^x(x) \) if necessary, we can assume that \(E_0^x = e\) (see Remark \ref{r: equivariance}). We will then write \( g:= F_0^x \), \(y=g(x)\) and \( E_n^y :=F_n^x \circ g^{-1}\). The sequence \( \{E_n^y\}_{n\geq 0}\) is then a \(c\)-line of expansion relative to the point \(y\), which explains the notation. Observe then that \(\norm{g}= d( E_0^x , F_0^x ) \). We denote \(x_m:= E_m^x (x) \) and \(y_n:= F_n^x(x)=E_n^y (y)\).

The map \(\widetilde{(E_m^x)^{-1}}\) is defined on \(B(x_m,r)\), where \(r\) is the constant given by Lemma \ref{l: distortion}, and Koebe's Lemma shows 
\[ \kappa \left( \widetilde{(E_m^x)^{-1}}, B(x_m, r/2) \right) \leq \kappa,\]
and 
\[  \widetilde{(E_m^x)^{-1}}\left( B(x_m, r/2) \right)\subset B\left(x, \frac{e^\kappa r }{2|DE_m^x (x)|}\right)   .\]

Since the map \(g\) is a composition of \(\norm{g}\) elements of \(S\), it has a univalent holomorphic extension \(\widetilde{g}\) defined on the ball \(B(x, r e^{-(\norm{g}-1) \overline{c}})\), that takes values in the ball \( B(y,r)\). Koebe's Lemma shows 
\[ \kappa \left( \widetilde{g}, B\left(x, r e^{-(\norm{g}-1) \overline{c}}/2 \right) \right) \leq \kappa.\]
In order that the composition \( \widetilde{g} \circ \widetilde{ (E_m^x)^{-1}}\) be defined on \(B(x_m,r/2)\) with distortion bounded by \(2\kappa\), a sufficient condition is then 
\begin{equation}\label{eq: condition x} \frac{e^\kappa r} {2 |DE_m^x (x)|} \leq \frac{re^{-(\norm{g}-1) \overline{c}}}{2}, \end{equation}
or equivalently
\begin{equation}\label{eq: condition x'} \log  |DE_m^x (x)|\geq  (\norm{g}-1) \overline{c} +\kappa .\end{equation}
In this case we have for every \(r'\leq r/2\) 
\[ \widetilde{g} \circ \widetilde{ (E_m^x)^{-1}} \left(B(x_m, r') \right) \subset  B(y, e^{2\kappa}r'  |D (g\circ (E_m^x)^{-1})(x_m) | ),\]
because \((E_m^x)_m\) is a \(c\)-line of expansion. Observe that condition \eqref{eq: condition x'} is fulfilled if 
\begin{equation} \label{eq: condition x''} 
m\geq \gamma_1 \norm{g} +\gamma_2 \text{ where } \gamma_1 = \overline{c}/c , \ \gamma_2 =\frac{\kappa-\overline{c}}{c}.\end{equation}

Similarly, the map  \(\widetilde{(E_n^y)^{-1}}\) is defined on the ball of radius \(r\), and we have  
\[  B\left(y, \frac{e^{-\kappa} r}{2 |DE_n^y (y)|}\right) \subset \widetilde{(E_n^y)^{-1}}\left( B(y_n, r/2) \right)   \] 
and 
\[ \kappa \left( \widetilde{(E_n^y)^{-1}}, B(y_n, r/2) \right) \leq \kappa.\]
 Then the extension \( \widetilde{E_n^y} \) is well-defined and univalent on the ball  \( B\left(y, \frac{e^{-\kappa} r}{2 |DE_n^y (y)|}\right) \) and its distortion is bounded by 
 \[ \kappa \left( \widetilde{E_n^y}, B\left(y, \frac{e^{-\kappa} r}{2 |DE_n^y (y)|}\right) \right) \leq \kappa.\]
In particular, for every \(r'\leq r/2\), we have, under the condition 
\begin{equation}\label{eq: condition yy} e^{2\kappa}r'  |D (g\circ (E_m^x)^{-1})(x_m) |\leq \frac{e^{-\kappa} r}{2  |DE_n^y (y)| } ,\end{equation}
or equivalently 
\begin{equation} \label{eq: condition y'} 3\kappa  + D_x(E_m^x, E_n^y\circ g)  \leq \log (r/2r')   ,\end{equation}
that the composition  \( \widetilde{E_n^y} \circ \widetilde{g} \circ \widetilde{ (E_m^x)^{-1}}\) is defined on the ball \( B(x_m, r') \) and its distortion is bounded by \(3\kappa\). Condition \eqref{eq: condition y'} is satisfied if \( \log (r/2r') \geq 3\kappa +\overline{c}\).

With this choice of \(r'>0\), we can then apply Lemma \ref{l:discreteness} with  \(\gamma_3=\gamma\) to get the conclusion.\end{proof}

\section{Gromov hyperbolicity of \(G\)} \label{s: hyperbolicity}

\begin{proposition}\label{p: hyperbolicity}
Let \( G\subset \text{Diff}^\omega({\bf S}^1) \) be a finitely generated subgroup which is expanding and locally discrete in the analytic category. Then, \(G\) is Gromov hyperbolic.  
\end{proposition}

\begin{proof}
We will use the following result \cite[Theorem 2.11]{Nekrashevych1} of Nekrashevych. 

Fix \(x\in {\bf S}^1\), and denote by \(\Gamma^x\) the {\bf directed graph} whose vertices are the elements of \(G\), and whose directed edges are the couples \( (g_0, g_1) \in  G ^2 \) such that 
\begin{equation} \label{eq: definition directed edges} D_x (g_0, g_1) \geq   \underline{c}/2 \text{ and } d(g_0, g_1) \leq 2,\end{equation}
where \(\underline{c} >0\) is the constant given by Remark \ref{r: uniformity}. The set \(G\) is equipped with the combinatorial metric \(d_{\Gamma^x}\) induced by the underlying undirected graph induced by \(\Gamma^x\): \(d_{\Gamma^x} (g_1,g_2)\) is the minimum number of undirected edges of \(\Gamma^x\) necessary to go from \(g_1\) to \(g_2\). Hence, the combination of Proposition \ref{p:line of expansion} and of \cite[Theorem 1.2.9]{Nekrashevych1} with the following constants \(\eta= \underline{c}/4\), \(\Delta=\overline{c}\),  show that \(\Gamma^x\) equipped with its distance \(d_{\Gamma^x}\) is Gromov hyperbolic. 

\vspace{0.3cm}

\noindent {\bf Claim: } \textit{The inclusion \((G, d_{\Gamma^x}) \rightarrow (G, d)  \) is a quasi-isometry}

\begin{proof}  From the definition of \(\Gamma_x\), we immediately have 
\begin{equation}\label{eq: bound 1}  d\leq 2 d_{\Gamma^x} .  \end{equation}
Let \( \{ g_1,g_2 \} \)  be an edge of \(G\), i.e. \(g_2\in Sg_1\). We can assume that \(D_x(g_1, g_2) \geq 0\) up to exchanging \(g_1\) and \(g_2\). If \(D_x(g_1,g_2) \geq \underline{c} /2 \), then the directed arrow \(g_1\rightarrow g_2\) belongs to \(\Gamma^x\). If not, let \(s\in S\) be an element such that \( D_x (g_1, sg_1) \geq \underline{c}\) given by Remark \ref{r: uniformity}. Then \(\Gamma^x\) contains the directed edges: \(g_1\rightarrow sg_1\) (by definition of \(s\)) and \(g_2\rightarrow sg_1\). So the \(d_{\Gamma^x}\)-distance between \(g_1\) and \(g_2\) in the graph is bounded by \(2\). In particular, we get 
\begin{equation}\label{eq: bound2} d_{\Gamma^x} \leq 2 d.\end{equation}
The claim follows from equations \eqref{eq: bound 1} and \eqref{eq: bound2}.
\end{proof}

The Proposition follows from the claim and the fact that Gromov hyperbolicity is a quasi-isometric invariant. \end{proof}

\section{The group \(G\) is virtually a Fuchsian group}  \label{s: Fuchsian group}
A consequence of Proposition \ref{p: hyperbolicity} is that we have a well-defined map 
\begin{equation}\label{def: map to boundary} \Omega : {\bf S}^1 \rightarrow \partial G\end{equation} 
which associates to a point \( x\in {\bf S}^1\) the equivalence class in \(\partial G\) of a \( \underline{c}\)-line of expansion at the point \(x\) (here \(\underline{c}\) is the constant appearing in Remark \ref{r: uniformity}). Indeed, Proposition \ref{p:line of expansion} shows that two such lines of expansion are at a bounded Hausdorff distance from each other. 

\begin{proposition} \label{p: boundary circle} The map \(\Omega : {\bf S}^1 \rightarrow \partial G\) is a finite covering.\end{proposition}

\begin{proof} The proof is organized as a sequence of claims. 

\vspace{0.2cm} 

\noindent {\bf Claim 0 --} \textit{The map \(\Omega : {\bf S}^1 \rightarrow \partial G\) is equivariant and continuous.}

\begin{proof} The equivariance is immediate. Let us prove the continuity at a point \(x\in {\bf S}^1\). Suppose \( \{E_m^x\}_m\) is a \( \underline{c}\)-line of expansion at \(x\) where \(m>0\) is the constant given by Remark \ref{r: uniformity}. Let \(m_0\) be a large integer. For \(y\in {\bf S}^1\) in a sufficiently small neighborhood of \(x\), denoting \( y_m := E_m^x (y)\), and recalling that \( E_m^x = s_m^x \circ E_{m-1}^x\) (see Definition \ref{def: line of expansion}), we have 
\[ \log |D s_m^x (y)|\geq \underline{c} /2 \text{ for every } 0\leq m\leq m_0.\]
Hence, one can define a \(\underline{c}/2\)-line of expansion \(\{E_m^y\} _{m\geq 0}\) relative to \(y\) by 
\[ E_m^y:= E_m^x \text{ if } m\leq m_0\]
and by taking for \( \{E_m^y\} _{m\geq m_0}\) a \(\underline{c}\)-line of expansion relative to \(y\). The two \( \underline{c}/2\)-lines of expansion \( \{E_m^x\}_m\) and \(\{E_m^y\}_m\) relative to \(x\) and \(y\) respectively coincide for \(m\leq m_0\) and converge respectively to \(\omega_x\) and \(\omega_y\). Hence, being \((\alpha,0)\)-quasi-geodesics for a constant \(\alpha\) depending only on \(\overline{c}\) and \(S\) (see Lemma \ref{l: quasi-geodesic}) their limit points \(\omega_x\) and \(\omega_y\) are close to each other in \(\partial G\). This proves continuity of \(\Omega\).\end{proof}

\vspace{0.2cm}

\noindent {\bf Claim 1 --} \textit{There are constants \(c,d  >0\) such that, for every \(x \in {\bf S}^1\),  any geodesic ray \( \{g_n\} _{n\geq 0}\) in \(G\) starting at \(g_0=e \) and tending to \( \omega \in \partial G\) satisfies 
\[  D_x (e,g_n) \geq c n -d \text{ for } n\leq  (\Omega (x), \omega)_e  \] 
and 
\[ D_x (e, g_n) \leq - cn +\overline{c} \cdot (\Omega (x), \omega)_e + d  \text{ for } n\geq  (\Omega (x),\omega)_e .\]}

\begin{proof}
Let \(\{E_l^x\}_{l\geq 0}\) and \(\{F_m^x\}_{m\geq 0}\) be \(\underline{c}\)-lines of expansion relative to \(x\) beginning at \(E_0^x = e\) and \(F_0^x= g_n\) respectively. By Proposition \ref{p:line of expansion}, there exist integers \(L\) and \(M\) such that 
\[ d( E_L^x , F_M ^{x}  ) \leq \gamma_3 .\]
Let \(\alpha\) be given by Lemma \ref{l: quasi-geodesic}, and let \(\beta =\gamma_3\). The \((\alpha,\beta)\)-quasi-geodesic triangle formed by the segments 
\[  \{g_s\}_{0\leq s\leq n}, \  \{E_l^x\} _{0\leq l\leq L} \text{ and } \{ F_m ^{x}\}_{0\leq m\leq M} ,\] 
is \(\delta\)-thin for a certain constant \(\delta\) depending only on \((\alpha, \beta)\) and the hyperbolicity constants of \((G,d)\). So for some integer \(N\), the segments \(\{ g_s\} _{0\leq s\leq (\Omega (x),\omega)_e}\) and  \(\{ g_s\} _{n \geq s \geq (\Omega (x),\omega)_e} \) are \(\delta\)-close to \(\underline{c}\)-segments of expansion relative to \(x\). The conclusion follows.
\end{proof}

\vspace{0.2cm} 

\noindent {\bf Claim 2 --} \textit{There is a number \(M\in {\bf N}^*\) such that the level subsets \( \Omega^{-1} (\omega) \), for \(\omega \in \partial G\), have cardinality less than \(M\). }

\begin{proof} Indeed, let \( x^k \in \Omega^{-1} (\omega)\), \(k=1,\ldots ,M\), be distinct points, and let \(\{g_n\}_n\) be a geodesic ray from \(e\) to \( \omega\in \partial G\). For every \(k=1,\ldots, M\), the sequence \(\{g_n\}\) is \(O(\delta)\)-close to a \(\underline{c}\)-line of expansion \(\{E_m^{x^k}\}_{m\geq 0}\) relative to \(x^k\) and starting at \(E_0^{x^k} =e\). 

Let \(n\) be a large integer. For every \(k=1,\ldots, M\), there exists \(m_k\) such that \( d( g_n, E_{m_k}^{x^k} ) = O (\delta)\).  Let \(r>0\) be the constant given by Lemma \ref{l: distortion}.  Denoting \( x_n^k= E_{m_k}^{x^k}  (x^k)\) for \(n\geq 0\) and \(k=1,\ldots, M\), Lemma \ref{l: distortion} shows that the map \( \widetilde{(E_{m_k}^{xk})^{-1}}\) extends as a univalent map defined on \( B(x_n^k , r/2) \) whose image lies in the ball \(B( x^k , \frac{re^{\kappa}}{2 |Dg_n(x^k)|} )\). In particular, there exists \( r'>0\) depending only on \(r\) and \(\delta\) such that, denoting \(y_n^k = g_n(x^k)\), the map \(g_n ^{-1} \) sends the interval \(B_{{\bf S}^1} ( y_n^k, r') \) inside the interval \( B_{{\bf S}^1} ( x^k , \frac{re^{\kappa}}{2 |Dg_n(x^k)|} )\).

For \(n\) large enough, the intervals \( B_{{\bf S}^1} ( x^k , \frac{re^{\kappa}}{2 |Dg_n(x^k)|} )\),  \(k=1,\ldots, M\), are disjoint, hence so are the intervals \(B_{{\bf S}^1} (x_n^k, r')\)'s. In particular, \(M\leq 1/r'\), and the claim follows. \end{proof}

Let \(\mathcal K\) be the set of compact subsets of \({\bf S}^1\) equipped with the topology induced by the Hausdorff distance.

\vspace{0.2cm}

\noindent {\bf Claim 3 --} \textit{ The map \(\omega \in \partial G \mapsto \Omega^{-1} (\omega ) \in \mathcal K\) is continuous.} 

\begin{proof} Let \(\{\omega ^k\}_k\) be a sequence of points of \(\partial G\) tending to \( \omega^\infty \in \partial G\). For every  \(k\in {\bf N}\), let \( \{g_n^k \}_{n\geq 0}\) be a geodesic ray tending to \(\omega^k\). Up to extracting if necessary, one can assume that for each \(n\), the sequence \( \{ g_n^k\}_{k\geq 0}\) is stationary,  namely for \(k\geq k(n)\), \(g_n^k= g_n^\infty\). The sequence \( \{ g_n^\infty\} _n\) is a geodesic ray tending to \(\omega^\infty\). Using the same notations as in Claim 1, the Hausdorff distance between \( \Omega^{-1} (\omega^k) \) and the set \( \{\log |Dg_n^k| \geq 0\} \) is less than \(\varepsilon\) for every \(n\geq n (\varepsilon)\), for every \(k\in {\bf N} \cup \{\infty\}\). Applying this to \( k \geq k(n(\varepsilon))\), we get that the Hausdorff distance between \(\Omega^{-1} (\omega^k ) \) and \(\Omega^{-1}(\omega^\infty)\) is less than \(2\varepsilon\), which proves the claim.  \end{proof}

\vspace{0.2cm}

\noindent {\bf Claim 4 --} \textit{The function \(k: \omega \in \partial G\mapsto |\Omega^{-1} (\omega)|\in {\bf N}\) is constant. }

\begin{proof} Notice that \(G\) cannot be virtually cyclic since otherwise its action on the circle could not be expanding. By \cite[Chapitre 8]{GdH}, it acts minimally on its boundary. Claim 3 shows that \(k\) is lower semi-continuous, and it is \(G\)-invariant. In particular,  the subset \(\{k=\min k\}\subset  \partial G\) is closed, non empty, and \(G\)-invariant. The action of \(G\) on \(\partial G\) being minimal, it must be the whole \(\partial G\), hence the conclusion holds. 
\end{proof}

Claims 2, 3 and 4 show that \( \Omega\) is a covering.  \end{proof}

\begin{corollary}\label{c: conjugation Fuchsian group}
The action of \(G\) on \({\bf S}^1\) is topologically conjugated to the action of a (cocompact) lattice of \(\widetilde{\text{PGL}_2^k({\bf R})} \) on the \(k\)-th covering \(\widetilde{{\bf R}P^1}^k\) for a certain integer \(k>0\). 
\end{corollary}

\begin{proof}
By Proposition \ref{p: boundary circle}, the action of \(G\) on \({\bf S}^1\) is topologically conjugate to a finite covering of its action on its boundary. This proves that the boundary of \(G\) is homeomorphic to the circle, hence the result follows from the convergence group theorem \cite{Gabai, CJ}. The lattice is cocompact since otherwise the group \(G\) would be virtually free, and its boundary would be a Cantor set.
\end{proof}

\section{Differentiable rigidity} \label{s: differentiable rigidity}


It is presumably well-known that Corollary \ref{c: conjugation Fuchsian group}, together with the differentiable rigidity theory developed by Ghys in \cite{Ghys_rigidite}, imply our main result, the Corollary \ref{c: classification}. However, this implication is not directly stated in this form in the litterature, so we provide a detailed proof below.

\begin{proposition}\label{p: invariant projective structure}
Let \( G\subset \text{Diff}^\omega({\bf S}^1) \) be a finitely generated subgroup which is expanding and locally discrete in the analytic category. Then \(G\) preserves an analytic \({\bf R}P^1 \)-structure.
\end{proposition}

\begin{proof}
We first observe that we can assume that the group \(G\) preserves the orientation on \({\bf S}^1\). Indeed, suppose that we know that the subgroup \(G^+\) of elements of \(G\) that preserves the orientation on \({\bf S}^1\) preserves an analytic projective structure \(\sigma\) on \({\bf S^1}\). If \(G^+=G\) we are done. If not, let \(g\) be an element of \(G\) that reverses the orientation. Since \(G\) is generated by \(G^+\) and \(g\), it suffices to prove that \(g\) preserves the projective structure \(\sigma\) as well. For this, let us consider the quadratic differential defined by \( S(g) := \{ g, z \} dz^2\), where \(\{g,z\}\) is the Schwarzian derivative \(\{g,z\}= \frac{D^3 g}{Dg}- \frac{3}{2} \left(\frac{D^2 g}{Dg}\right)^2\) computed in projective coordinates of \(\sigma\). The cocycle relations satisfied by the Schwarzian derivative shows that \(S(g)\) is a well-defined quadratic differential on the circle, namely it is independant of the chosen projective coordinates. We need to prove that \(S(g) =0\) identically. For every \(f\in G^+\), there exists \(h\in G^+\) such that \( g\circ f = h\circ g\). The cocycle relation satisfied by the Schwarzian derivative, together with the fact that both \( f\) and \(h\) preserve \(\sigma\), implies  \( S (g) = f^* S(g)\), hence proving that \(S(g)\) is invariant by the whole group \(G^+\). In particular, if \(S(g)\) does not vanish identically, the volume \(\sqrt{|S(g)|}\) is \(G^+\)-invariant, which contradicts the fact that \(G^+\) is expanding on the support of \(S(g)\). 

So from now on we will assume that \(G\) preserves orientation on \({\bf S}^1\). Denote by \( \phi : \partial G \rightarrow {\bf R}P^1\) a homeomorphism that conjugates the \(G\)-action on its boundary to the action of a Fuchsian group \(\Gamma\subset \text{PGL}_2({\bf R})\) on \({\bf R}P^1\) (see Corollary \ref{c: conjugation Fuchsian group} for the existence of \(\phi\)). We can assume that \(\phi\) preserves orientation, so that the lattice \(\Gamma\) is indeed contained in \(\text{PSL} _2({\bf R})\). We denote by \( \rho :G\rightarrow \Gamma\) the representation that satisfies 
\[ \phi ( g \omega ) = \rho (g) \phi (\omega) \text{ for every } g\in G, \ \omega \in \partial G.\]

Let \( U\subset  {\bf S}^1 \times {\bf R}P^1 \) be the complement of the graph of \(\Phi = \phi \circ \Omega\), namely  the set of points \( (x, z) \in {\bf S}^1 \times {\bf R}P^1\) such that \( z \neq \Phi(x)\). Consider the analytic action of \(G\) on \( U \times {\bf R}\) defined by: 
\begin{equation}\label{eq: action} g (x, z , t ) = (g(x), \rho(g)(z) , t + \log |Dg(x)| ) .\end{equation}

\vspace{0.2cm} 

\noindent {\bf Claim 0 --} \textit{The action \eqref{eq: action} is free, proper discontinuous and cocompact. Hence the quotient of \( U\) by \(G\) is a closed analytic \(3\)-manifold \(M\).  } 

\vspace{0.2cm} 

\textit{Freeness.} Assume that there exists a point \( (x,z,t)\) which is fixed by an element \(g\in G\). The theory of hyperbolic groups (see e.g. \cite[Chapitre 8, \textsection 3]{GdH}) tells us that \(g\) is either of finite order, or a hyperbolic element, in which case the sequence \(\{g^n\}_{n\in{\bf N}}\)  is a quasi-geodesic. Assume by contradiction that \(g\) is hyperbolic. It has a fixed point \(x\) on \({\bf S}^1\) having a derivative equal to \(1\) at \(x\) -- recall that we assume the orientation is preserved by \(G\) -- so all the numbers \( D_x (e, g^n) \) are zero. This contradicts the claim 1 of the proof of Proposition \ref{p: boundary circle} since the sequence \(\{g^n\}_{n\geq 0}\) is a quasi-geodesic. This contradiction shows that \(g\) has finite order in \(G\). Its image \(\rho(g)\) is an element of \(\text{PSL}_2({\bf R})\) of finite order that must be the identity since it fixes the point \(z\in {\bf R}P^1\).  Hence, \(g\) lies in the kernel of \(\rho\) which is a cyclic subgroup of \(G\) acting freely on \({\bf S}^1\). Since \(g\) fixes the point \(x\), it is therefore the identity map. Hence the action is free as claimed.

\vspace{0.1cm} 

\textit{Proper discontinuity.}  Any compact set in \(U\) is contained in some compact set of the sort
\begin{equation}\label{eq: compact set}  K:= \{ (x,z,t)\in U\times {\bf R}  \ | \ (\Omega(x), \phi^ {-1} (z) )_e \leq C \text{ and } |t|\leq T\} ,\end{equation}
for some constants \(C,T\). We must prove that there is only a finite number of elements \(g\) of \(G\) such that \(gK\cap K\neq \emptyset\). Let \(g\in G\), and suppose that for some \((x,z,t)\in K\) we have \( g(x,z,t)\in K\). Since both 
\((\Omega(x), z)_e\) and \( (g \Omega(x), g\phi^{-1} (z))_e\) are bounded by \(C\),  \(g\) must lie at a distance from a geodesic between \(\phi^{-1} (z) \) and \(\Omega(x)\) bounded by some constant depending only on \(C\) and the constants of hyperbolicity of \((G,d)\).  But a geodesic from \(\phi^{-1} (z)\) to \(\Omega(x)\) lies at a finite distant from a \(\underline{c}\)-line of expansion between \(\phi^{-1}(z)\) and \(\Omega(x)\) relative to \(x\), the Hausdorff distance being bounded by some constant depending only on \(\underline{c}\) and \(G,S\). Hence, we have \( |\log |Dg(x)| | \geq \text{cst} \norm{g} + \text{cst}\) for some positive constants depending only on \(\underline{c}\) and \(G\). However \( |\log |Dg(x)| |\leq 2T\) since both \((x,z,t)\) and \(g(x,z,t)\) lie in \(K\). This proves that the norm of \(g\) is bounded, hence there is only a finite number of \(g\in G\) such that \(gK\cap K\neq \emptyset\), and the properness of the action of \(G\) on \(U\times {\bf R}\) follows.

\vspace{0.1cm} 

\textit{Cocompactness.} The cocompactness of the action of \(G\) on \(U\times {\bf R}\) follows from cohomological reasons, but it is instructive to prove it by hands. We will prove that any point \((x,z,t)\in U\times {\bf R}\) can be sent by an element of \(G\) to a point \((x'',z'',t'')\) belonging to some compact set defined by equation \eqref{eq: compact set}, for some constants \(C,T\) that depend only on \(G\).

Let us first find \( g\in G\) such that \(\rho(g) (\Phi( x)) \) and \( \rho(g) (z)\) are separated by a constant that depends only on \(G\). Let \( \{g_n\}_{n\geq 0}\) be a geodesic ray on \(G\) tending to \(\Omega(x)\). A consequence of Lemma \ref{l: distortion} is that, if \(n\) is large enough so that \( | Dg_n(x) | \) is larger than the inverse of the distance between \(x\) and \(\Phi^{-1} (z)\), the distance between \(g_n(x) \) and the set \(\Phi^{-1} (\rho(g_n) z)\) is larger than some constant depending only on \(G\). In particular, the points \(\Phi (g_n(x) ) \) and \( \rho(g_n) (z)\) are separated by a constant that depends only on \(G\). Similarly, suppose that a geodesic ray \(\{g_n\}_{n\geq 0} \) tends to \(\phi^{-1} (z)\) when \(n \) tends to infinity. Then, the sequence \(\{g_n \}_{n\geq 0}\) lies at a finite distant from a line of expansion at \(z\) (for the \(\rho\)-action of \(G\) on \({\bf R}P^1\) given by \(\rho\), which is locally discrete and expansive), and the same argument applies: namely, for \(n\) large enough, the distance between \(\rho(g_n)(\Omega(x))\) and \(\rho(g_n) (z)\) is bounded from below by some constant depending only on \(G\). So we are done.

Now write \( (x',z',t') = g(x,z,t) \), with \( \Omega(x') \) and \(z'\) separated by a constant depending only on \(G\). If \(t'\leq 0\), then the \(t''\)-variable of the point \( (x'', z'', t'') := E^{x'}_n (x',z',t') \) grows linearly with \(n\), while keeping the distance between \( \Phi (x'') \) and \( z''\) separated by a constant depending only on \(G\). Hence for some \(n\) the point \((x'', z'', t'')\) belongs to the compact set defined in equation  \eqref{eq: compact set} for some constants \(C,T\) depending only on \(G\). If \(t''>0\), let \(\{g_n\}_{n\geq 0}\) be a geodesic ray beginning at \(g_0=e\) and tending to \(\phi^{-1} (z')\) when \(n\) tends to \(\infty\). Because \(\Phi (x') \) and \(z'\) are separated by some constant depending only on \(G\), the Gromov product \( (\Omega(x'), \phi^{-1} (z') )_e\) is bounded by some constant depending only on \(G\) as well, and the claim 1 of the proof of the Proposition \ref{p: boundary circle} shows that \( \log |D g_n (x') | \) decreases linearly when \(n\) tends to \(\infty\). Letting \( (x'',z'',t''):= g_n(x',z',t')\), and reasoning as before, we infer that the Gromov product  \( ( \Omega(x''), \phi^{-1} (z'') ) _e\) remains bounded by some constant depending only on \(G\), and for some \(n\), the \(t''\)-coordinates enters in some fixed interval of the form \([-T,T]\) for some constant \(T\) depending only on \(G\) as well. The conclusion follows. 

\vspace{0.2cm} 

\noindent {\bf Claim 1 --} \textit{The flow on \(M\) induced by the non singular vector fields \(V:=\frac{\partial}{\partial t}\) is Anosov, and its weak stable  foliation \(\mathcal F^+\) is given by the equation \(dx=0\).} 

\begin{proof}
Let \( \Delta\subset U\times {\bf R}\) be a fundamental domain for the \(G\)-action given by \eqref{eq: action}. Let \( (x,z,t)\in \Delta\) be some point, and \( s\) be a negative real number. There exists an element \(g_s\) of \(G\) such that \( g_s(x,z,t) \in \Delta\). This element \(g_s\) can be chosen as the \(n\)-th term of an expansive line at the point \(x\) beginning at \(e\), with \(n\) growing linearly with \(-s\) and constants depending only on \(G\). Hence \( \log |Dg_s (x)| \) (resp. \( \log |D\rho(g_s) (z)| \)) increases linearly to \(+\infty\) with \(-s\) (resp. decreases linearly to \(-\infty\) with \(-s\)). As a consequence: when \(s\) tends to \(+\infty\), the flow \(\exp (sV)\) contracts exponentially a metric on the bundle \( T\mathcal F^- /{\bf R} V\), where \(\mathcal F^-\) is the foliation  defined on the covering \(U\) by \(z=\text{cst}\), whereas it expands exponentially a metric on the foliation \(\mathcal F^{++}\) defined on the covering \(U\) by \( (x,t)=\text{cst}\). The conclusion follows.
\end{proof}

The weak unstable foliation \(\mathcal F^+\) defined on the covering \(U\) by \( x=\text{cst}\) is analytic, hence a Theorem of Ghys \cite[Th\'eor\`eme 4.1]{Ghys_rigidite} shows that  \(\mathcal F^+\) has an analytic transverse projective structure. This latter lifts to a transverse projective structure on the foliation of \(U\) defined by the submersion \( x: U\rightarrow {\bf S}^1\), which is invariant by \(G\), hence gives birth to an analytic projective structure on \({\bf S}^1\) which is invariant by \(G\). The proof of Proposition \ref{p: invariant projective structure} is complete. \end{proof}

\begin{corollary}\label{c: classification}
A finitely generated subgroup of \(\text{Diff}^\omega ({\bf S}^1) \) which  is expanding and locally discrete in the analytic category  is analytically conjugated to a uniform lattice in \( \widetilde{\text{PGL}}^k _2({\bf R}) \) acting on the \(k\)-th covering of \({\bf R}P^1\) for a certain integer \(k>0\).
\end{corollary}

\begin{proof}  By Proposition \ref{p: invariant projective structure}, the group \(G\) preserves an analytic  \({\bf R}P^1\)-structure on \({\bf S}^1\). Let  \( D : {\bf R} \rightarrow {\bf R}P^1\) be a developing map for this structure: this is an analytic local diffeomorphism which is equivariant with respect to an element \(A\in \text{PGL}_2({\bf R})\), namely we have 
\begin{equation}\label{eq: equivariance} D(\tilde{x} +1) = A D(\tilde{x}) \text{ for every } \tilde{x} \in {\bf R}.\end{equation}
The group \(G\) lifts to a subgroup  \(\widetilde{G}\subset \text{Diff}^\omega ({\bf R})\) that commutes with the translation \( \tilde{x}\mapsto \tilde{x}+1\).  Since the \({\bf R}P^1\)-structure is invariant by \(G\), there exists a representation \( \tilde{\rho'}: \widetilde{G}\rightarrow \text{PGL}_2({\bf R}) \) which is such that 
\begin{equation}\label{eq: invariance of projective structure} D\circ \tilde{g}= \tilde{\rho'} (\tilde{g} ) \circ D.\end{equation}

We then have \( A=\tilde{\rho'} (\tilde{x}\mapsto \tilde{x}+1) \). Hence \(A\) commutes with \(\tilde{\rho'}\left( \widetilde{G}\right) \). We will prove that \(A=Id\). Assume by contradiction that \(A\) has no fixed point on \({\bf R}P^1\). In this case \(\widetilde{G}\) is contained in a conjugate of \(\text{PO}_2({\bf R})\), and in particular preserves a measure on \({\bf R}P^1\) with an analytic density. Its preimage by \(D\) is invariant by \(\widetilde{G}\), in particular by the translation \(\tilde{x}\mapsto \tilde{x} +1\), hence produces on \({\bf S}^1\) a measure invariant by \(G\) that has an analytic density as well. This contradicts the expanding property for \(G\). Assume by contradiction that \(A\) has some fixed point but is not the identity. Then \(D^{-1} (\text{Fix} (A) ) \) is a discrete \(\widetilde{G}\)-invariant subset of \({\bf R}\), and projects in \({\bf S}^1\) to a \(G\)-invariant finite orbit. This is a contradiction since the image of this latter by the map \(\Omega\) would  be a finite \(G\)-orbit in \(\partial G\). The only remaining possibility is that \(A= Id\) as claimed. 

In particular, the representation \(\tilde{\rho'}\) induces a representation \(\rho' : G\rightarrow \text{PGL} _2 ({\bf R})\),  and \eqref{eq: equivariance} shows that  \(D\) induces a finite covering from \({\bf S}^1\) to \({\bf R}P^1\) which is \(\rho'\)-equivariant. Since \(G\) is a finitely generated, locally discrete and expanding subgroup of \(\text{Diff}^\omega ({\bf S}^1)\), the same is true for the image of \(\rho'\), hence this latter is a uniform lattice in \(\text{PGL}_2 ({\bf R})\). The result follows.\end{proof}

\end{document}